\documentclass[a4paper, 10pt, twoside, notitlepage]{amsart}

\usepackage{amsmath,amscd}
\usepackage{amssymb}
\usepackage{amsthm}
\usepackage{comment}
\usepackage{graphicx, xcolor}

\usepackage{mathrsfs}
\usepackage[ocgcolorlinks, linkcolor=blue]{hyperref}

\usepackage{bm}
\usepackage{bbm}
\usepackage{url}

\usepackage[utf8]{inputenc}
\usepackage{mathtools,amssymb}
\usepackage{esint}
\usepackage{tikz}
\usepackage{dsfont}
\usepackage{relsize}
\usepackage{url}
\usepackage{xcolor}
\usepackage{graphicx}
\usepackage{mathrsfs}
\usepackage[shortlabels]{enumitem}
\usepackage{lineno}
\usepackage{amsmath}
\usepackage{enumitem}
\usepackage{amsthm} 
\usepackage{verbatim}
\usepackage{dsfont}
\numberwithin{equation}{section}

\allowdisplaybreaks

\mathtoolsset{showonlyrefs}

\graphicspath{{images/}}

\newtheorem{theorem}{Theorem}[section]
\newtheorem{lemma}[theorem]{Lemma}

\newtheorem{proposition}[theorem]{Proposition}

\newtheorem{remark}[theorem]{Remark}

\title[Fractional anisotropic Calder\'on problem]{The fractional anisotropic Calder\'on problem for a nonlocal parabolic equation on closed Riemannian manifolds}

\author[Y.-H.~Lin]{Yi-Hsuan Lin}
\address{Department of Applied Mathematics, National Yang Ming Chiao Tung University, Hsinchu, Taiwan \& Fakult{\"a}t f{\"u}r Mathematik, University of Duisburg-Essen, Essen, Germany}
\email{yihsuanlin3@gmail.com}

\newcommand{\R}{{\mathbb R}}

\newcommand{\N}{{\mathbb N}}

\newcommand {\p} {\partial}
\newcommand {\dVg} {\mathsf{dV}_g}

\newcommand{\LC}{\left(}
\newcommand{\RC}{\right)}
\newcommand{\wt}{\widetilde}

\newcommand{\norm}[1]{\lVert #1 \rVert}
\newcommand{\abs}[1]{\left\lvert #1 \right\rvert}
\DeclareMathOperator{\supp}{supp} 

\begin{document}

	\maketitle
	\begin{abstract}
		We consider the fractional anisotropic Calder\'on problem for the nonlocal parabolic equation $(\partial_t -\Delta_g)^s u=f$ ($0<s<1$) on closed Riemannian manifolds. More concretely, we can determine the Riemannian manifold $(M,g)$ up to isometry by using the local source-to-solution map in an arbitrarily small open cylinder in the spacetime domain. This can be regarded as a nonlocal analog of the anisotropic Calder\'on problem in the parabolic setting. We also study several useful properties for nonlocal parabolic operators by using comprehensive spectrum analysis with semigroup theory.

		\medskip
		
		

	\end{abstract}

	\tableofcontents
	
	\section{Introduction}\label{sec: introduction}
	
	In 1980, A. P. Calder\'on \cite{calderon} proposed a problem: Is that possible to recover the (anisotropic) conductivity in a medium by using its voltage and current measurements on the boundary? This problem is often referred to as the \emph{anisotropic Calder\'on problem}, which can be reformulated in the geometrical setting, see \cite{LU1989}. In other words, can one recover $(M,g)$ by using the boundary measurements (usually, researchers utilize either the Dirichlet-to-Neumann map or the Neumann-to-Dirichlet map)? In the works \cite{LU1989,LU01,LLS2020_Poisson}, the authors studied the anisotropic Calder\'on problem on real-analytic manifolds with real-analytic metrics. However, without using analyticity assumptions, the anisotropic Calder\'on problem is a well-known long-standing open problem in the inverse problems community for the spatial dimension $n\geq 3$. 
	
	Recently, in the works \cite{feizmohammadi2021fractional,feizmohammadi2021fractional_closed,FKU2024calder}, the authors investigated the fractional anisotropic Calder\'on problem on closed connected smooth Riemannian manifolds, which they used the local source-to-solution map as measurement so that they can determine the manifold up to isometry. This type of result is different from that of their local counterparts. On the one hand, the fractional anisotropic Calder\'on problem can be resolved. On the other hand, the result holds for any spatial dimension $n\geq2$, without additional analytic assumptions on manifolds $(M,g)$.
	The study of fractional inverse problems has been an attractive research field. Thanks to the pioneering work \cite{GSU20}, the authors proved two important materials: The strong uniqueness and Runge approximation. Motivated by the work \cite{GSU20}, there are many articles studied fractional type inverse problems, such as \cite{GRSU20,cekic2020calderon,CLL2017simultaneously,CMR20,CMRU20,GLX,CGRU2023reduction,GU2021calder,HL19_monotone1,HL20_monotone2,LL2020inverse,LL2022inverse,LLR2019calder,LLU2022para,LLU2023calder,lin2020monotonicity,lin2023determining,ruland2023revisiting,ruland2018exponential,RS17,KLW2021calder,KRZ-2023,KLZ-2022,LZ2024approximation,LZ2023unique,LTZ24_optimal,LTZ24_well-posed,LNZ_Calderon} and references therein.\\

	The mathematical model in this paper is given as follows.
	Let $(M,g)$ be a smooth closed Riemannian manifold of dimension $\dim M\geq 2$, and $-\Delta_g$ be the positive Laplace-Beltrami operator defined on $M$, i.e., $g=g(x)\in C^\infty(M)$ satisfies the classical ellipticity condition (we do not address more in this paper). A closed manifold is compact without boundary. Given $s\in (0,1)$ and $T>0$, let 
	\[
	\mathcal{H}:= \p_t -\Delta_g , 
	\]
	and we consider the fractional parabolic equation
	\begin{equation}\label{equ: frac para}
		\begin{cases}
			\mathcal{H}^s u = f &\text{ in }M\times (-T,T) \\
			u=0 &\text{ in }M\times \{t\leq -T\},
		\end{cases}
	\end{equation}
	where the fractional parabolic operator $\mathcal{H}^s$ can be defined by the spectrum theory (see Section \ref{sec: prel}). The well-posedness of \eqref{equ: frac para} holds for any $f$ in a suitable function space with the compatibility condition 
	\begin{equation}\label{compatibility condition}
		f=0 \text{ in } M\times \{t\leq -T\}.
	\end{equation}
	Let $\mathcal{O}\subset M$ be an open set, by the well-posedness of \eqref{equ: frac para}, we can define the local source-to-solution map $S_{M,g,\mathcal{O}}$ of \eqref{equ: frac para} via 
	\begin{equation}\label{local source-to-solution map}
		S_{M,g,\mathcal{O}}: C^\infty_c(\mathcal{O}\times (-T,T)) \ni f \mapsto \left. u^f \right|_{\mathcal{O}\times (-T,T)} \in \mathbf{H}^s(\mathcal{O}\times (-T,T)),
	\end{equation}
	where $u^f\in \mathbf{H}^s(M\times \R)$ is the solution to \eqref{equ: frac para}. Here $\mathbf{H}^s(M\times \R)$ is a suitable function space given Section \ref{sec: prel}.

	Recall that $-\Delta_g$ is a nonnegative linear operator on the Hilbert space $L^2(M)$. Since $M$ is a connected closed smooth Riemannian, it is known that $-\Delta_g$ possesses a discrete spectrum, which can be denoted as an eigenpairs $\left\{ \LC \lambda_k,\phi_k\RC\right\}_{k\geq 0}$, such that $0=\lambda_0 <\lambda_1 \leq \lambda_2 \leq \ldots \nearrow \infty$ (up to normalization) and $\left\{  \phi_k \right\}_{k\geq 0}$ is an orthonormal basis in $L^2(M)$. Notice that these eigenvalues $\lambda_k$ may not be distinct. The zeroth eigenfunction $\phi_0\equiv 1$ in $M$. This implies that 
	\begin{equation}\label{ortho eigen}
		\LC \phi_k , 1 \RC_{L^2(M)}=\LC \phi_k , \phi_0\RC_{L^2(M)}=0, \quad \text{for all }k\in \N.
	\end{equation}
	Furthermore, it is known that the fractional Laplace-Beltrami operator can be defined by 
	\begin{equation}
		\LC -\Delta_g \RC^s u (x)=\sum_{k=1}^\infty \lambda_k^s \LC u,\phi_k \RC_{L^2(M)}\phi_k(x) \text{ for }x \in M,
	\end{equation}
	where we use $\lambda_0=0$, and $\LC  \varphi , \psi \RC_{L^2(D)}:=\int_{D} \varphi\overline{\psi}\, \dVg$, for all $\varphi,\psi \in L^2(M)$. Here $D$ can be any subset of $M$, and $\dVg$ is the Riemannian volume element of $(M,g)$ with $\overline{\psi}$ being the complex conjugate of the (complex) function $\psi$. Similar $L^2$ integral formula notation will be utilized for any subsets in the space-time domain $M\times \R$.		
	Note that $\LC (-\Delta_g)^s u , 1 \RC_{L^2(M)}=0$, for any $u\in \mathrm{Dom}((-\Delta_g)^s)$, where we used \eqref{ortho eigen} and $\mathrm{Dom}((-\Delta_g)^s):=\big\{ u\in L^2(M): \, \sum_{k=0}^\infty \lambda_k^{2s} | ( u,\phi_k)_{L^2(M)}|^2 <\infty\big\}$.

	Motivated by the above observations, the nonlocal parabolic operator $\mathcal{H}^s$ can be defined spectrally. More precisely, given any function $u=u(x,t)\in L^2(M\times \R)$, one can write
	\begin{equation}
		\begin{split}
			u(x,t)= \sum_{k=0}^\infty u_k(t)\phi_k(x) =\frac{1}{\sqrt{2\pi}} \int_{\R} \sum_{k=0}^\infty  \widehat{u_k}(\rho)\phi_k (x)e^{\mathsf{i}t\rho}\, dt,
		\end{split}
	\end{equation}
	where $\mathsf{i}=\sqrt{-1}$. Here $u_k(t)$ and $\widehat{u_k}(\rho)$ (the time-Fourier transform of $u_k(t)$) are given by 
	\begin{equation}\label{u_k and its Fourier transf}
		\begin{split}
			u_k(t):=\LC u(\cdot,t),\phi_k(\cdot) \RC_{L^2(M)}\quad \text{and}\quad 	\widehat{u_k}(\rho):=\frac{1}{\sqrt{2\pi}} \int_{\R} e^{-\mathsf{i}\rho t}u_k (t)\, dt,
		\end{split}
	\end{equation}
	respectively, for $k\geq 0$. The fractional parabolic operator $\mathcal{H}^s$ is (formally) defined by the pairing 
	\begin{equation}
		\left\langle  \mathcal{H}^s u , v \right\rangle= \int_{\R} \sum_{k=0}^\infty \left( \mathsf{i}\rho +\lambda_k\right)^{s} \widehat{u_k}(\rho) \overline{\widehat{v_k}(\rho)}\, d\rho ,
	\end{equation}
	where we will justify the above pairing in Section \ref{sec: prel} under appropriate sense with suitable function spaces.
	We will introduce more properties of the operator $\mathcal{H}^s$ in Section \ref{sec: prel}. Our main result to determine $(M,g)$ (up to isometry) by using the local source-to-solution map \eqref{local source-to-solution map} is stated in the forthcoming theorem:

	\begin{theorem}\label{thm: main}
		Given $s\in (0,1)$, let $(M_1,g_1)$ and $(M_2,g_2)$ be closed connected smooth Riemannian manifolds of $\dim M_1= \dim M_2\geq 2$. Let $\mathcal{O}_j \subset M_j$ be a nonempty open sets (for $j=1,2$) such that 
		\begin{equation}\label{same local information}
			\LC \mathcal{O}_1, \left. g_1 \right|_{\mathcal{O}_1} \RC  =\LC \mathcal{O}_2, \left. g_2 \right|_{\mathcal{O}_2} \RC :=\LC \mathcal{O},  g  \RC .
		\end{equation}
		Suppose that 
		\begin{equation}\label{same local source to solution}
			S_{M_1,g_1,\mathcal{O}_1}(f)= S_{M_2,g_2, \mathcal{O}_2}(f), \text{ for any }f\in C^\infty_c(\mathcal{O}\times (-T,T)),
		\end{equation}
		where $S_{M_j,g_j,\mathcal{O}_j}$ denotes the local source-to-solution map of 
		\begin{equation}
			\begin{cases}
				\LC \p_t -\Delta_{g_j}\RC^s u_j =f &\text{ in }M\times (-T,T), \\
				u_j=0  &\text{ in }M\times \{t\leq -T\},
			\end{cases}
		\end{equation}
		for $j=1,2$. Then there exists a diffeomorphism $\Phi: M_1\to M_2$ such that $\Phi^\ast g_2=g_1$ in $M_1$ $\Phi=\mathrm{Id}$ (the identity map) on $\mathcal{O}$.
	\end{theorem}
	
	The preceding theorem was demonstrated in \cite{feizmohammadi2021fractional} for the elliptic case. Since $\LC \p_t -\Delta_g \RC^s$ is a space-like nonlocal operator (as $g=g(x)$), we could expect a similar result for this nonlocal parabolic operator.\\

	\noindent\textbf{Strategy.}	To prove Theorem \ref{thm: main}, our strategy is to demonstrate that the conditions \eqref{same local information} and \eqref{same local source to solution} can be used to deduce the identity  
	\begin{equation}\label{same heat kernel}
		e^{-\tau \mathcal{L}_1}(x,z)=e^{-\tau \mathcal{L}_2}(x,z), \text{ for }x,z \in \mathcal{O} \text{ and }\tau>0,
	\end{equation}
	where $\mathcal{L}_j:=-\Delta_{g_j}$, for $j=1,2$. If the condition \eqref{same heat kernel} holds, then one can apply \cite[Theorem 1.5]{feizmohammadi2021fractional} to conclude Theorem \ref{thm: main} as we wish. For this purpose, we need to analyze the function $e^{-\tau \mathcal{H}} f$ for any $f\in C^\infty_c(\mathcal{O}\times (-T,T))$ carefully (see Sections \ref{sec: prel} and \ref{sec: inverse problem}).\\

	\noindent\textbf{Organization of the article.}	The paper is organized as follows. In Section \ref{sec: prel}, we provide a rigorous definition of the nonlocal parabolic operator via the spectrum theory and related function spaces will be introduced. Furthermore, we demonstrate that the future data $u|_{\{t\geq T\}}$ of the solution $u$ to \eqref{equ: frac para} will not affect the solution $u|_{\{-T\leq t\leq T \}}$. In Section \ref{sec: well-posed}, we prove the well-posedness of \eqref{equ: frac para}, which also shows that the operator $\mathcal{H}^s$ is a space-like nonlocal operator. Finally, the proof of Theorem \ref{thm: main} will be given in Section \ref{sec: inverse problem}, which can be derived by showing the conditions \eqref{same local information} and \eqref{same local source to solution} ensure the identity \eqref{same heat kernel}.

	\section{The nonlocal parabolic operator on closed Riemannian manifolds}\label{sec: prel}

	Given $a\in \R$, let us denote 
	\begin{equation}
		\mathbf{H}^a(M\times \R):= \left\{ u\in L^2(M\times \R): \, \norm{u}_{\mathbf{H}^{a}(M\times \R)}<\infty \right\},
	\end{equation}
	where 
	\[
	\norm{u}_{\mathbf{H}^{a}(M\times \R)}= \bigg( \int_{\R} \sum_{k=0}^\infty  \LC 1 +\abs{\mathsf{i}\rho +\lambda_k}^2 \RC^{a/2} \left| \widehat{u_k}(\rho)\right|^2 \, d\rho \bigg)^{1/2}.
	\]
	Here 
	\begin{equation}
		\abs{\mathsf{i}\rho + \lambda_k}^2 =\rho^2 + \lambda_k^2 , \text{ for }k\in \N\cup \{0\}.
	\end{equation}
	It is obvious that $\mathbf{H}^s(M\times \R)\subset L^2(M\times \R)$.
	Moreover, for any open set $O\subset M\times \R$, let us denote 
	\begin{equation}
		\begin{split}
			\mathbf{H}^a (O)&:=\left\{ u|_{O} : \, u\in \mathbf{H}^s(M\times \R)\right\},\\
			\widetilde{\mathbf{H}}^a (O) & :=\text{closure of }C^\infty_c(O) \text{ in }\mathbf{H}^a(M\times \R).
		\end{split}
	\end{equation}
	By the duality, there holds that 
	\begin{equation}
		\begin{split}
			(\mathbf{H}^a(O))^{\ast}=\widetilde{\mathbf{H}}^{-a}(O) \quad \text{and}\quad (\widetilde{\mathbf{H}}^a(O))^\ast = \mathbf{H}^{-a}(O).
		\end{split}
	\end{equation}
	We also write 
	\begin{equation}
		\mathbf{H}^s_F =\left\{ u\in \mathbf{H}^s(M\times \R): \, \supp(u)\subset F\right\},
	\end{equation}
	for any closed set $F\subset M\times \R$.
	
	Let us first review several useful properties for the nonlocal operator $\mathcal{H}^s$. Given $s\in (0,1)$, let us define the domain of $\mathcal{H}^s$ by
	\begin{equation}
		\begin{split}
			\mathrm{Dom}(\mathcal{H}^s)= \mathbf{H}^{2s}(M\times \R)
		\end{split}
	\end{equation}
	where 
	Here $\LC \lambda_k\RC_{k\geq 0}$ is the eigenvalue of $-\Delta_g$ in $M$ introduced in the previous section. We denote $\mathbf{H}^{-s}(M\times \R)$ as the dual space of $\mathbf{H}^s(M\times \R)$, and it is not hard to check that 
	\begin{equation}\label{mapping property}
		\begin{split}
			\mathcal{H}^s: \mathbf{H}^s (M\times \R) & \to \mathbf{H}^{-s}(M\times \R), \\
			\mathcal{H}^s: \mathbf{H}^{s/2} (M\times \R)&\to L^2(M\times \R).
		\end{split}
	\end{equation}  
	The first relation can be seen by using the classical pairing argument. More concretely, one has that 
	\begin{equation}\label{pairing}
		\LC \mathcal{H}^s u , v \RC_{L^2(M\times \R)}= \int_{\R} \sum_{k=0}^\infty \left( \mathsf{i}\rho +\lambda_k\right)^{s} \widehat{u_k}(\rho) \overline{\widehat{v_k}(\rho)}\, d\rho ,
	\end{equation}
	for any $u,v \in C^\infty(M\times \R)$. Then a density argument ($C^\infty(M\times \R)$ is dense in $\mathbf{H}^s(M\times \R)$ with respect to the norm $\norm{\cdot}_{\mathbf{H}^s}$) yields that \eqref{pairing} holds, so that the mapping property \eqref{mapping property} is valid in the following form 
	\begin{equation}
		\begin{split}
			\left\langle \mathcal{H}^s u , v \right\rangle_{\mathbf{H}^{-s}\times \mathbf{H}^{s}} = \int_{\R} \sum_{k=0}^\infty \left( \mathsf{i}\rho +\lambda_k\right)^{s} \widehat{u_k}(\rho) \overline{\widehat{v_k}(\rho)}\, d\rho  
		\end{split}
	\end{equation}
	and we use $\norm{u}_{\mathbf{H}^s}=\norm{u}_{\mathbf{H}^s(M\times \R)}$ for the simplicity. This shows the first property in \eqref{mapping property}. For the second property, it is easy to see by using the definition of $\mathcal{H}^s$. More generally, one can also derive that $\mathcal{H}^b$ satisfies 
	\begin{equation}
		\mathcal{H}^b : \mathbf{H}^a (M\times \R) \to \mathbf{H}^{a-2b}(M\times \R),
	\end{equation} 
	for any $a\in \R$ and $b\geq 0$.

	On the other hand, we have the following integration-by-parts formula that 
	\begin{equation}\label{inte-by-part}
		\LC \mathcal{H}^s u , v \RC_{\mathbf{H}^{-s}\times \mathbf{H}^{s}} =  \big( \mathcal{H}^{s/2} u , \mathcal{H}^{s/2}_\ast v \big)_{L^2(M\times \R)},
	\end{equation}
	where $\mathcal{H}^{s}_\ast:= \LC -\p_t -\Delta_g \RC^s$ denotes the adjoint operator of $\mathcal{H}^s$, which is also a time-reversal operator of $\mathcal{H}^s$. Similar to the definition of $\mathcal{H}^s$, one can define $\mathcal{H}^s_\ast$ via 
	\begin{equation}
		\left\langle \mathcal{H}^s_\ast v , u \right\rangle_{\mathbf{H}^{-s}\times \mathbf{H}^s}  = \int_{\R}\sum_{k=0}^\infty \LC -\mathsf{i}\rho +\lambda_k\RC^s \widehat{v_k}(\rho) \overline{\widehat{u_k}(\rho)}\, d\rho,
	\end{equation}
	where $\widehat{u_k}(\rho)$ and $\widehat{v_k}(\rho)$ are given by \eqref{u_k and its Fourier transf}.

	Recall that $\LC \phi_k\RC_{k\geq 0}$ is an orthonormal basis in $L^2(M)$, and we can write the semigroup $\left\{ e^{-\tau \mathcal{L}}\right\}_{\tau \geq 0}$ generated by $\mathcal{L}:=-\Delta_g$ on $M$ as 
	\begin{equation}\label{heat generator}
		\begin{split}
			e^{-\tau \mathcal{L}} \varphi(x) =\sum_{k=0}^\infty e^{-\tau \lambda_k}  \LC \varphi, \phi_k \RC_{L^2(M)}\phi_k=\int_M e^{-\tau \mathcal{L}} (x,z)\varphi(z)\, dz,
		\end{split}
	\end{equation}
	for any $\varphi\in L^2(M)$ and for $k\geq 0$. Here $e^{-\tau \mathcal{L}}(x,z)$ is the heat kernel of $\p_t -\Delta_g=\p_t +\mathcal{L}$ on $M$, which is symmetric and nonnegative that
	\begin{equation}
		e^{-\tau \mathcal{L}}(x,z)=e^{-\tau \mathcal{L}} (z,x)\geq 0, \text{ for all }x,z\in M.
	\end{equation}
	Moreover, $e^{-\tau \mathcal{L}}(x,z)$ possesses the Gaussian upper bound
	\begin{equation}\label{Gaussian upper bound}
		e^{-\tau \mathcal{L}}(x,z) \leq \frac{C}{\tau^{n/2}}e^{-\frac{\abs{x-z}^2}{c\tau}}, \text{ for all }x, z\in M \text{ and }\tau >0,
	\end{equation}
	where $c,C>0$ are some positive constants (cf. \cite[Theorem 4.7, page 17]{SY_lecture}). 
	On the other hand, via the condition \eqref{ortho eigen}, observe that 
	\begin{equation}\label{unit heat kernel condition}
		e^{-\tau \mathcal{L}}1 (x)=1, \text{ for all }x\in M \text{ and }\tau >0,
	\end{equation}
	whenever $(M,g)$ is a closed Riemannian manifold.

	Since $\p_t$ and $-\Delta_g$ are commutable, we have $e^{-\tau \mathcal{H}}= e^{-\tau \mathcal{L}} \circ e^{-\tau \p_t}$ and 
	\begin{equation}\label{heat evolute semigroup}
		\begin{split}
			e^{-\tau \mathcal{H}}u(x,t)=e^{-\tau \mathcal{L}}\LC e^{-\tau \p_t} u \RC =e^{-\tau \mathcal{L}}(u(\cdot, t-\tau))(x)=\int_{M}e^{-\tau \mathcal{L}} (x,z)u(z,t-\tau)\, dz
		\end{split}
	\end{equation}
	in the sense that 
	\begin{equation}
		\begin{split}
			\LC  e^{-\tau \mathcal{H}}u,v \RC_{L^2(M\times \R)} &=\int_{\R}\sum_{k=0}^\infty e^{-\tau (\mathsf{i}\rho +\lambda_k )} \widehat{u_k}(\rho)\overline{\widehat{v_k}(\rho)}\, d\rho \\
			&= \int_{\R}\sum_{k=0}^\infty e^{-\tau \lambda_k}u_k (t-\tau) v_k(t)\, dt \\
			&= \int_{\R}\int_M \int_M e^{-\tau \mathcal{L}} (x,z)u(z,t-\tau) v(x,t)\, dzdxdt.
		\end{split}
	\end{equation}
	Recalling that$(M,g)$ is a closed connected smooth Riemannian manifold, by using the formula \eqref{pairing}, it is not hard to see 
	\begin{equation}\label{expoenential decay}
		\left\| e^{-\tau \mathcal{H}} u \right\|_{L^2(M\times \R)} \leq e^{-c_0 \tau } \norm{u}_{L^2(M\times \R)}, \quad  \text{for }\tau \geq 0,
	\end{equation}
	for some constant $c_0>0$. 
	
	With the above expressions and formulae, we can show an alternative formulation used to characterize $\mathcal{H}^s$ by the heat semigroup method.
	
	\begin{lemma}\label{Lemma: Hs semigroup}
		Given $s\in (0,1)$, for any $u\in \mathbf{H}^s(M\times \R)$, there holds that 
		\begin{equation}\label{heat extension}
			\mathcal{H}^su=\frac{1}{\Gamma(-s)} \int_0^\infty \LC e^{-\tau \mathcal{H}}u-u\RC \frac{d\tau}{\tau^{1+s}}
		\end{equation}
		in the sense that 
		\begin{equation}
			\left\langle \mathcal{H}^su, v \right\rangle_{\mathbf{H}^{-s}\times \mathbf{H}^{s}} =\frac{1}{\Gamma(-s)} \int_0^\infty\big( \LC e^{-\tau \mathcal{H}}u, v \RC_{L^2(M\times \R)} - \LC u,v \RC_{L^2(M\times \R)}\big) \frac{d\tau}{\tau^{1+s}},
		\end{equation}
		for any $v\in \mathbf{H}^s(M\times \R)$.
	\end{lemma}

	\begin{proof}
		By using the formula 
		\begin{equation}
			\LC \mathsf{i}\rho + \lambda_k\RC^s =\frac{1}{\Gamma(-s)}\int_0^\infty \big( e^{-\tau (\mathsf{i}\rho +\lambda_k)}-1 \big) \frac{d\tau}{\tau^{1+s}},
		\end{equation}
		where the integral converges absolutely (see the proof of \cite[Lemma 2.1]{BDLCS2021harnack}). Combining with the formula \eqref{pairing}, there holds that 
		\begin{equation}
			\begin{split}
				\left\langle \mathcal{H}^s u,v \right\rangle_{\mathbf{H}^{-s}\times \mathbf{H}^s}= \int_{\R} \sum_{k=0}^\infty \bigg[ \frac{1}{\Gamma(-s)}\int_0^\infty \big(  e^{-\tau (\mathsf{i}\rho +\lambda_k)} -1 \big) \frac{d\tau}{\tau^{1+s}} \bigg] \widehat{u_k}(\rho)\overline{\widehat{v_k}(\rho)}\, d\rho ,
			\end{split}
		\end{equation}
		which proves the assertion.
	\end{proof}

	Let us refer readers to the \emph{Balakrishnan formula} (see \cite[equation (9.63), page 285]{Samko02}), which can be used to define the general fractional order of operators.
	Last but not least, we can derive an integral representation formula of $\mathcal{H}^s$.
	
	\begin{proposition}[Representation formula]
		We have the integral formula of $\mathcal{H}^s$ by 
		\begin{equation}\label{representation formula of H^s}
			\begin{split}
				\mathcal{H}^s u(x,t)=\int_0^\infty \int_{M}  \mathcal{K}_s (x,z;\tau)\LC u(z,t-\tau) -u(x,t)\RC dz d\tau 
			\end{split}
		\end{equation}
		which belong to $\mathbf{H}^{-s}(M\times \R)$, for any $u\in \mathbf{H}^s(M\times \R)$, where 
		\begin{equation}\label{kernel nonlocal}
			\mathcal{K}_s(x,z;\tau):= \frac{1}{\Gamma(-s)}\frac{e^{-\tau \mathcal{L}}(x,z)}{\tau^{1+s}}, \text{ for }x,z\in M \text{ and }\tau >0.
		\end{equation}
	\end{proposition}
	
	\begin{proof}
		Note that from the formulae \eqref{heat generator} and \eqref{unit heat kernel condition}, then there holds that 
		\begin{equation}\label{heat kernel =1}
			1 =e^{-\tau \mathcal{L}}1(x)= \int_M e^{-\tau \mathcal{L}}(x,z) \, dz, \text{ for all }x\in M.
		\end{equation}
		Now, by using \eqref{heat evolute semigroup} and \eqref{heat extension}, we have 
		\begin{equation}
			\begin{split}
				&\quad \, \mathcal{H}^su(x,t)\\
				&= \frac{1}{\Gamma(-s)}\int_0^\infty \bigg( \int_M e^{-\tau \mathcal{L}} (x,z)u(z,t-\tau)\, dz -u(x,t) \bigg) \, \frac{d\tau}{\tau^{1+s}} \\
				&= \frac{1}{\Gamma(-s)}\int_0^\infty \bigg( \int_M e^{-\tau \mathcal{L}}(x,z)u(z,t-\tau)\, dz -\underbrace{\bigg( \int_M e^{-\tau \mathcal{L}}(x,z)\, dz \bigg)}_{=1\text{ by \eqref{heat kernel =1}}} u(x,t) \bigg) \, \frac{d\tau}{\tau^{1+s}}  \\
				&=\frac{1}{\Gamma(-s)}\int_0^\infty  \int_M e^{-\tau \mathcal{L}} (x,z) (u(z,t-\tau)-u(x,t)) \, dz \frac{d\tau}{\tau^{1+s}},
			\end{split}
		\end{equation}
		where we use the Fubini theorem in the last identity. Finally, using \eqref{kernel nonlocal}, this proves the assertion.
	\end{proof}

	\begin{remark}
		If $u$ is time-independent, i.e., $u(x,t)\equiv u(x)$, it is easy to see that $\mathcal{H}^su(x)=(-\Delta_g)^su(x)$ by using the formula \eqref{representation formula of H^s}. However, from the initial data of the problem \eqref{equ: frac para}, one can not simply take the solution $u(x,t)$ as a time-independent function so that one can apply the existing results given in \cite{feizmohammadi2021fractional} to show Theorem \ref{thm: main}.
	\end{remark}

	Motivated by \cite{LLR2019calder}, we have the following lemma.
	
	\begin{lemma}\label{Lemma: auxiliary}
		Let $\chi_{A}(t)$ be the characteristic function of $A\subset \R$, for any $u\in \mathbf{H}^s(M\times \R)$, then there holds that 
		\begin{enumerate}[(i)]
			\item $\chi_{A}(t)u(x,t)\in \mathbf{H}^s(M\times \R)$, for any interval $A\subset \R$.
			
			\item $\mathcal{H}^s u(x,t)= \mathcal{H}^s\LC \chi_{(-\infty,T](t)}u(x,t)\RC$ for $(x,t)\in \Omega\times (-T,T)$.
		\end{enumerate}
	\end{lemma}
	
	\begin{proof}
		Since $\chi_{A}(t)$ is a space-independent function, the first statement follows by utilizing a similar approach given in \cite[Section 3]{LLR2019calder}. For the second statement, note that $ \chi_{(-\infty,T](t)}u\in \mathbf{H}^s(M\times \R)$ provided that $u\in \mathbf{H}^s(M\times \R)$. Then the identity holds by using the integral representation \eqref{representation formula of H^s}. More specifically, one may compute that 
		\begin{equation}
			\begin{split}
				&\quad \, \mathcal{H}^su(x,t)\\
				&=\int_0^\infty \int_{M}  \mathcal{K}_s (x,z;\tau)\LC u(z,t-\tau) -u(x,t)\RC dz d\tau \\
				&= \int_0^\infty \int_{M}  \mathcal{K}_s (x,z;\tau)\big(  \big( \chi_{(-\infty, T]}(t)u\big) (z,t-\tau) - \big( \chi_{(-\infty, T]}(t)u\big)  (x,t)\big) \, dz d\tau 
			\end{split}
		\end{equation}
		for $(x,t)\in M\times (-T,T)$, and $u\in \mathbf{H}^s(M\times \R)$, and the last identity holds since $\tau \in (0,\infty)$. This completes the proof.
	\end{proof}
	
	The above lemma states that the value $\left. \mathcal{H}^su \right|_{M\times (-T,T)}$ will not be affected by the information $u|_{\{t\geq T\}}$. In other words, the value $\mathcal{H}^su(x,t)$ depends only on the past, not the future. Similar properties as in Lemma \ref{Lemma: auxiliary} hold for the adjoint operator $\mathcal{H}^s_\ast$, and the value of $\left.\mathcal{H}^s_\ast u\right|_{M\times [-T,T]}$ will not be affected by the past time $\{t\leq -T\}$.

	\section{The well-posedness}\label{sec: well-posed}
	
	We next turn to study the well-posedness of \eqref{equ: frac para}.
	As explained in \cite{LLR2019calder}, the operator $\mathcal{H}^s$ is a space-time coupled operator, which is the space-like nonlocal operator.
	Thus, combining with the above relation \eqref{inte-by-part}, one can define the bilinear form associated with \eqref{equ: frac para} via 
	\begin{equation}
		B_g(u,v):= \big( \mathcal{H}^{s/2} u , \mathcal{H}^{s/2}_\ast v \big)_{L^2(M\times \R)}.
	\end{equation}
	then similar to the arguments as in \cite{LLR2019calder,BS2022calderon}, one can check that $B_g(\cdot, \cdot)$ is bounded and coercive in the space $\mathbf{H}^s(M\times \R)$. It is worth mentioning that this argument was first found in the work \cite{LLR2019calder}, which also demonstrates that $\mathcal{H}^s$ is a space-like nonlocal operator.

	\begin{lemma}[Well-posedness]
		Let $(M,g)$ be a closed connected smooth Riemannian manifold with $\dim M\geq 2$. Then given $f\in \mathbf{H}^{-s}(M\times \R)$ fulfilling \eqref{compatibility condition}, there exists a unique solution $u\in H^s(M\times \R)$ solving \eqref{equ: frac para}.	 
	\end{lemma}

	\begin{proof}
		Notice that even $\mathcal{H}^s$ is a fractional parabolic operator, we will use the Lax-Milgram theorem to prove the result. To simplify the notation, let us denote 
		$$
		u_T(x,t)=\chi_{[-T,T]}(t)u(x,t),
		$$ which belongs to $\mathbf{H}^s(M\times \R)$ as $u\in \mathbf{H}^s(M\times \R)$ (see Lemma \ref{Lemma: auxiliary}).
		On one hand, for the boundedness, by using the H\"older inequality, we have 
		\begin{equation}\label{boundedness}
			\begin{split}
				&\quad \,  \left| B_g (u_T,v)\right| \\
				& = \bigg| \int_{\R}\sum_{k=0}^\infty \left( \mathsf{i}\rho + \lambda_k \right) ^s \widehat{u_{T,k}}(\rho)\overline{ \widehat{v_k}(\rho)}\, d\rho \bigg| \\
				&\leq \bigg( \int_{\R}\sum_{k=0}^\infty  \left| \mathsf{i}\rho + \lambda_k \right| ^s \left| \widehat{u_{T,k}}(\rho) \right|^2 \, d\rho \bigg)^{1/2} \bigg( \int_{\R}\sum_{k=0}^\infty  \left| \mathsf{i}\rho + \lambda_k \right| ^s \left| \widehat{v_k}(\rho) \right|^2 \, d\rho \bigg)^{1/2}\\
				&\leq \norm{u_T}_{\mathbf{H}^s(M\times \R)}  \norm{v}_{\mathbf{H}^s(M\times \R)},
			\end{split}
		\end{equation}
		for any $v\in \mathbf{H}^s_{M\times [-T,T]}$, where $u_{T,k}$ and $\widehat{u_{T,k}}$ are the functions given by \eqref{u_k and its Fourier transf} when the function $u$ is replaced by $u_T$.

		On the other hand, for the coercive, we can see that 
		\begin{equation}\label{coercive 1}
			\begin{split}
				B_g(u_T,u_T)  &=  \int_{\R}\sum_{k=0}^\infty \left( \mathsf{i}\rho +\lambda_k \right)^s  \big| \widehat{u_{T,k}}(\rho)\big|^2 \, d\rho  \\
				&=  \int_{\R} \bigg[ \sum_{k=0}^\infty \LC \rho^2 + \lambda_k^2\RC^{s/2}\LC \cos (s\theta_k) + \mathsf{i}\sin(s\theta_k)\RC  \bigg] \big| \widehat{u_{T,k}}(\rho)\big|^2\, d\rho  ,
			\end{split}
		\end{equation}
		where $\theta_k =\tan^{-1}\big(\frac{\rho}{\lambda_k}\big)$ for $k\in \N$. Note that the sine function is odd, and $\big| \widehat{u_{T,k}}(\rho)\big|^2 =\big| \widehat{u_{T,k}}(-\rho)\big|^2 $ for any $\rho \in \R$, then there holds 
		$$
		\int_{\R} \sum_{k=0}^\infty \LC \rho^2 + \lambda_k^2 \RC^{s/2}\sin (s\theta_k )  \big| \widehat{u_{T,k}}(\rho)\big|^2\, d\rho =0.
		$$
		Meanwhile, since $\LC \phi_k \RC_{k\geq 0}$ is an orthonormal basis in $L^2(M)$, we can write $u_T(x,t)=\sum_{k=0}^\infty u_{T,k}(t)\phi_k(x)$ such that 
		$$
		\norm{u_T(\cdot, t)}^2_{L^2(M)}= \sum_{k=0}^\infty \big| u_{T,k}(t) \big|^2, \text{ for } t\in \R,
		$$
		where we used $\norm{\phi_k}_{L^2(M)}=1$, for all $k\geq 0$.
		Inserting the above identities into \eqref{coercive 1}, we can obtain 
		\begin{equation}\label{coercive 2}
			\begin{split}
				\left| B_g (u_T,u_T)\right| & =\int_{\R} \sum_{k=0}^\infty \LC \rho^2 +\lambda_k^2 \RC^{s/2} \cos (s\theta_k) \big| \widehat{u_{T,k}}(\rho) \big|^2 \, d\rho \\
				&\geq  \int_{\R} \sum_{k=0}^\infty\LC \rho ^2 +\lambda_k^2 \RC^{s/2} \cos (s\pi /2) \big| \widehat{u_{T,k}}(\rho)\big|^2 \, d\rho\\
				&\geq c_s \norm{u_T}_{\mathbf{H}^s(M\times \R)}^2,
			\end{split}
		\end{equation}
		for some constant $c_s>0$ depending only on $s\in (0,1)$.
		
		Therefore, for any $f\in C^\infty_c (M\times (-T,T))$, by using \eqref{boundedness} and \eqref{coercive 2}, the Lax-Milgram theorem yields that 
		\[
		B_g (u_T,v)=\LC f, v \RC_{L^2(M\times \R)},
		\]
		admits a unique solution $u_T \in \mathbf{H}^s(M\times \R)$ to the initial value problem \eqref{equ: frac para} with $u_T=0$ in $M\times \{t\leq -T\}$. This proves the assertion.	
	\end{proof}

	\begin{remark}
		With the above well-posedness at hand, the local source-to-solution map \eqref{local source-to-solution map} is automatically well-defined. Moreover, since the solution $u=0$ in $\{t\leq -T\}$ and $u_T=u$ in $M\times [-T,T]$, then it is possible to extend the measurement in the whole spacetime cylinder $\mathcal{O}\times \R$.
		
	\end{remark}

	\section{The Calder\'on problem}\label{sec: inverse problem}

	\subsection{Auxiliary lemmas}

	It is different to the nonlocal elliptic case as in \cite{feizmohammadi2021fractional}, the function $e^{-\tau \mathcal{H}}u$ is not a solution to any (local) ''parabolic equations", but we can still find a useful equation for the function $e^{-\tau \mathcal{H}}u$. The next lemma is similar to \cite[Lemma 3.1]{LLU2022para} in the Euclidean case.
	
	\begin{lemma}\label{Lemma: double heat equation}
		Let $u\in \mathbf{H}^s(M\times \R)$, and consider the function 
		\begin{equation}\label{tilde u}
			\wt u(x,t,\tau):= e^{-\tau \mathcal{H}}u(x,t), \text{ in }M\times \R\times (0,\infty).
		\end{equation} 
		Then $\wt u$ is a solution to 
		\begin{equation}\label{U equation}
			\begin{cases}
				\LC \p_t + \p_{\tau}-\Delta_{g}\RC  \wt u=0 &\text{ in }M\times \R\times (0,\infty), \\
				\wt u(x,t,0)=u(x,t) &\text{ in }M\times \R.
			\end{cases}
		\end{equation}
		Moreover, $\wt u(x,-T,\tau)=0$ for any $x\in M$ and $\tau \in (0,\infty)$.
		
	\end{lemma}

	\begin{proof}
		By using the representation formula \eqref{heat evolute semigroup} of $\wt u(x,t,\tau)$, we have 
		$$
		\wt u(x,t,\tau)=\int_{M}e^{-\tau \mathcal{L}}(x,z)u(z,t-\tau)\, dz,
		$$
		so that 
		\begin{align}\label{comp 1}
			\begin{split}
				\LC \p_\tau -\Delta_g  \RC \wt u
				&=\underbrace{ \int_{M} \left[\LC \p_\tau -\Delta_g \RC e^{-\tau \mathcal{L}}(x,z)  \right] u(t-\tau, z)\, dz}_{=0\text{ since }e^{-\tau \mathcal{L}}(x,z)\text{ is the heat kernel}} \\
				&\quad \,  +\int_{M}e^{-\tau \mathcal{L}}(x,z)\p_\tau \LC u(z,t-\tau)\RC dz \\
				&=\int_{M}e^{-\tau \mathcal{L}}(x,z)\p_\tau \LC u(t-\tau,z)\RC dz , 
			\end{split}
		\end{align}
		for $(x,t,\tau)\in M\times \R\times (0,\infty)$, where we used that $e^{-\tau \mathcal{L}}(x,z)$ is the heat kernel of $\p_{\tau}-\Delta_g$ on $M$. 
		Interchanging $\p_\tau$ and $\p_t$ in the right hand side of \eqref{comp 1}, one has 
		\begin{align*}
			\begin{split}
				\LC \p_\tau -\Delta_g \RC \wt u=-\p_t\LC \int_{M}e^{-\tau \mathcal{L}}(x,z) u(z,t-\tau) \, dz\RC=-\p_t \wt u  \text{ in }M\times \R\times (0,\infty), 
			\end{split}
		\end{align*}
		which shows the first equation of \eqref{U equation} holds.
		Moreover, it is easy to check that 
		\begin{align*}
			\wt u(x,t,0)=\lim_{\tau \to 0}\int_{M}e^{-\tau \mathcal{L}}(x,z)u(z,t-\tau)\,  dz=u(x,t), \text{ for }(x,t)\in M\times \R.
		\end{align*}		
		Now, given any $\tau >0$, we can directly see that
		\begin{align*}
			\wt u (x,-T,\tau)=\int_{M}e^{-\tau \mathcal{L}}(x,z)\underbrace{u_T(z,-T-\tau)}_{=0 \text{ for }\tau>0}\,  dz=0, \text{ for }(x,\tau)\in  M\times  (0,\infty).
		\end{align*}
		This concludes the proof.
	\end{proof}

	Let us reformulate the initial value problem \eqref{equ: frac para} into a source problem, which helps us eliminate the initial condition in \eqref{equ: frac para}. As we discussed in Section \ref{sec: well-posed}, the solution $u$ of \eqref{equ: frac para} in $\Omega_T$ is the same as the solution $u_T(x,t)=\chi_{(-T,T)}(t)u(x,t)$. Moreover, with the representation formula \eqref{representation formula of H^s} and the pass time condition in \eqref{equ: frac para} at hand, we can compute that 
	\begin{equation}\label{past memory=0}
		\begin{split}
			&\quad \, \left. \mathcal{H}^s u(x,t) \right|_{t\leq -T}\\
			&=\int_0^\infty \int_{M}  \mathcal{K}_s (x,z;\tau)\LC u(z,t-\tau) -u(x,t)\RC dz d\tau \bigg|_{\{ t\leq -T\}}\\
			&=0 \quad \text{in} \quad M,
		\end{split}
	\end{equation}
	since $u(x,t)=u(z,t-\tau)=0$ for $t\leq -T$ and for all $\tau \in (0,\infty)$. This also shows that the condition \eqref{compatibility condition} must be satisfied. In Lemma \ref{Lemma: auxiliary}, we have explained that the future data of $u|_{\{t\geq T\}}$ will not affect the solution $u$ of \eqref{equ: frac para} in the spacetime domain $M\times (-T,T)$. Hence, there would be infinitely many functions defined globally on $M\times \R$ solving the equation \eqref{equ: frac para} by varying the function $u|_{\{t\geq T\}}$.
	However, combining with \eqref{past memory=0}, let us adopt the notion for the ''uniqueness" that \eqref{equ: frac para} admits a unique solution $u_T\in \mathbf{H}^s_{M\times[-T,T]}$. In other words, the equation \eqref{equ: frac para} can be rewritten in terms of 
	\begin{equation}\label{source problem}
		\mathcal{H}^s u_T = f \text{ in }M\times (-T,T),
	\end{equation}
	for any $f\in C_c^\infty(\mathcal{O}\times (-T,T))$.

	By using the previous discussions, we can write the solution $u_T$ in \eqref{source problem} of the form 
	\begin{equation}\label{inverse of nonlocal para}
		u_T = \mathcal{H}^{-s}f \text{ in }M\times (-T,T),
	\end{equation}
	and we aim to write down a more explicit expression formula of $\mathcal{H}^{-s}$.
	
	\begin{lemma}\label{Lemma: inverse Hs}
		Given $0<s<1$, let $f=f(x,t)\in C^\infty_c(M\times (-T,T))$, then the solution $u_T$ of \eqref{source problem} is given by 
		\begin{equation}\label{inverse of Hs}
			\begin{split}
				u_T & = \frac{1}{\Gamma(s)}\int_0^\infty e^{-\tau \mathcal{H}} f \, \frac{d\tau}{\tau^{1-s}}\\
				&=\int_0^\infty \int_M \mathcal{K}_{-s}(x,z;\tau) f(z,t-\tau)\, dzd\tau ,
			\end{split}
		\end{equation}
		where $\mathcal{K}_{-s}(x,z,\tau)$ is the kernel function defined by \eqref{kernel nonlocal} whenever the exponent $s$ is replaced by $-s$.
	\end{lemma}
	
	\begin{proof}
		The proof can be seen by using the formula of the Gamma function that 
		\begin{equation}
			\LC \mathsf{i}\rho + \lambda_k\RC^{-s}=\frac{1}{\Gamma(s)}\int_0^\infty e^{-\tau (\mathsf{i}\rho +\lambda_k)}\, \frac{d\tau}{\tau^{1-s}}.
		\end{equation}
		The derivation in the first identity of \eqref{inverse of Hs} is similar to the proof of Lemma \ref{Lemma: Hs semigroup}. 
		For the second equality in \eqref{inverse of Hs}, one simply uses the heat kernel $e^{-\tau \mathcal{L}}(x,z)$ combining with the factor $\frac{1}{\Gamma(s)}\frac{1}{\tau^{1-s}}$. This proves the assertion.
	\end{proof}
	
	\begin{remark}
		The first identity of \eqref{inverse of Hs} will play an essential role in the proof of Theorem \ref{thm: main}.
	\end{remark}
	
	\subsection{Proof of Theorem \ref{thm: main}}

	Now, we can prove Theorem \ref{thm: main}.
	
	\begin{proof}[Proof of Theorem \ref{thm: main}]
		Let $\omega_1,\omega_2\subset \mathcal{O}$ be nonempty open sets such that $\overline{\omega_1}\cap \overline{\omega_2}=\emptyset$. Given an arbitrary $f\in C^\infty_c(\omega_1\times (-T,T))$, the condition \eqref{same local information} implies that 
		\begin{equation}
			\LC \p_t -\Delta_{g_1}\RC^m f =\LC \p_t -\Delta_{g_2}\RC^m f =\LC \p_t -\Delta_g\RC^m f \text{ on }\omega_1 \times (-T,T),
		\end{equation}
		and $\Delta^m_gf\in C^\infty_c(\omega_1\times (-T,T))$, for any $m\in \N$. Meanwhile, using \eqref{same local source to solution}, there holds that 
		\begin{equation}\label{s-t-s_0}
			\left.\mathcal{H}_{1}^{-s} \LC \LC \p_t - \Delta_g\RC ^m f \RC \right|_{\mathcal{O}\times (-T,T)} = \left.	\mathcal{H}_{2}^{-s} \LC \LC\p_t - \Delta_g\RC^m f \RC  \right|_{\mathcal{O}\times (-T,T)},
		\end{equation}
		for any $m\in \N$, where $\mathcal{H}_{j}:= \p_t -\Delta_{g_j}$, for $j=1,2$. Note that 
		$$
		w_j(x,t):=\chi_{[-T,T]}(t)\mathcal{H}_{j}^{-s} \LC \LC \p_t - \Delta_{g_j}\RC ^m f \RC(x,t)
		$$ 
		is the solution to 
		\begin{equation}
			\mathcal{H}_j^s w_j = \LC \p_t - \Delta_{g}\RC ^m f\in C^\infty_c(\omega_1 \times (-T,T)) \quad \text{in }M\times \R,
		\end{equation}			
		for $j=1,2$ and for any $m\in \N$. Then the identity \eqref{s-t-s_0} is equivalent to 
		\begin{equation}\label{s-t-s_1}
			\left.\mathcal{H}_{1}^{-s} \LC \LC \p_t - \Delta_g\RC ^m f \RC \right|_{\mathcal{O}\times \R} = \left.	\mathcal{H}_{2}^{-s} \LC \LC\p_t - \Delta_g\RC^m f \RC  \right|_{\mathcal{O}\times \R},
		\end{equation}
		where we use the same technique by multiplying the time cutoff function $\chi_{[-T,T]}(t)$, which was shown in Section \ref{sec: well-posed}. This makes that the solution $w_j$ are supported in $M\times [-T,T]$, which are zero in the sets $M\times \LC \{t\leq -T\} \cup \{t\geq T\}\RC$, for $j=1,2$. Let us emphasize again that the future information of these solutions in the set $M\times \{t\geq T\}$ will not affect the solution in the set $M\times (-T,T)$, and hence we simply take the solution to be zero in the future time domain.
		
		By using Lemma \ref{Lemma: inverse Hs}, the above identity yields that 
		\begin{equation}\label{pf of thm_1}
			\int_0^\infty \LC e^{-\tau \mathcal{H}_1} -e^{-\tau \mathcal{H}_2}\RC  \LC \LC \p_t -\Delta_g\RC^m f \RC   \frac{d\tau}{\tau^{1-s}}=0 ,
		\end{equation}
		for $(x,t)\in \mathcal{O}\times \R$ and $m\in \N$. Furthermore, for $f\in C^\infty_c(\omega_1\times (-T,T))$ and $\frac{n+1}{2}<\ell \in \N$, the Sobolev embedding and \eqref{expoenential decay} imply 
		\begin{equation}\label{pf of thm_2}
			\begin{split}
				&\quad \, \left\| e^{-\tau \mathcal{H}_j}\LC \p_t -\Delta_g \RC^m f \right\|_{L^\infty(M_j\times \R)} \\
				& \leq C 	\left\| e^{-\tau \mathcal{H}_j}\LC \p_t -\Delta_g \RC^m f \right\|_{H^\ell(M_j\times \R)} \\
				&\leq  C 	\big\| e^{-\tau \mathcal{H}_j} \LC 1-\p_t^2 -\Delta_{g_j}\RC^\ell \LC \p_t -\Delta_g \RC^m f \big\|_{L^2(M_j\times \R)}  \\
				&\leq C e^{-c_0\tau } 	\big\| \LC 1-\p_t^2 -\Delta_{g_j}\RC^\ell \LC \p_t -\Delta_g \RC^m f \big\|_{L^2(M_j\times \R)}
			\end{split}
		\end{equation} 
		which ensures that the the integral in \eqref{pf of thm_1} convergence uniformly for $(x,t)\in \mathcal{O}\times \R$, for any $m\in \N$. Notice that in the right-hand side of \eqref{pf of thm_2}, it is a local computation, for any smooth function $f\in C^\infty_c(\omega_1\times (-T,T))$.

		Similar to \cite{LLU2022para}, using the commutative property $e^{-\tau \mathcal{H}_j} \mathcal{H}_j^m = \mathcal{H}_j^me^{-\tau \mathcal{H}_j} $, for all $\tau \geq 0 $ on $\mathrm{Dom}\LC \mathcal{H}_j^m\RC$, we can obtain
		\begin{equation}\label{pf of thm_3}
			\begin{split}
				e^{-\tau \mathcal{H}_j}\LC \LC \p_t -\Delta_g \RC^m f \RC= \LC \p_t -\Delta_g\RC^m \LC e^{-\tau \mathcal{H}_j} f \RC \underbrace{=\LC -\p_\tau \RC^m \LC e^{-\tau \mathcal{H}_j} f\RC}_{\text{By \eqref{U equation}}},
			\end{split}
		\end{equation}
		for any $m\in \N$, where $e^{-\tau \mathcal{H}_j} f$ solves \eqref{U equation} as $g=g_j$, for $j=1,2$. Inserting \eqref{pf of thm_3} into \eqref{pf of thm_1}, there holds that 
		\begin{equation}\label{pf of thm_4}
			\begin{split}
				\int_0^\infty \p_\tau^m \LC \LC e^{-\tau \mathcal{H}_1}  -e^{-\tau \mathcal{H}_2}  \RC f\RC (x,t) \frac{d\tau}{\tau^{1-s}}=0,
			\end{split}
		\end{equation}
		for $(x,t)\in \mathcal{O}\times \R$ and $m\in \N$.
		Meanwhile, as $(x,t)\in \omega_2 \times (-T,T)$, we can integrate by parts \eqref{pf of thm_4} $m$-times. Note that there are no end-point contributions while doing this integration by parts. More concretely, as showed in \cite{feizmohammadi2021fractional}, using the exponential decay \eqref{pf of thm_2} as $\tau \to \infty$ and the Gaussian upper bound \eqref{Gaussian upper bound} for the heat kernel $e^{-\tau \mathcal{L}}(x,z)$ as $\tau \to 0$, one can guarantee that 
		\begin{equation}\label{pf of thm_5}
			\begin{split}
				\lim_{\tau \to \infty} &\bigg| \p_\tau^\ell  \LC \LC e^{-\tau \mathcal{H}_1}  -e^{-\tau \mathcal{H}_2}  \RC f\RC (x,t)\frac{1}{\tau^{k-s}}\bigg|\\
				=\lim_{\tau \to 0}&\bigg| \p_\tau^\ell  \LC \LC e^{-\tau \mathcal{H}_1}  -e^{-\tau \mathcal{H}_2}  \RC f\RC (x,t) \frac{1}{\tau^{k-s}}\bigg|=0,
			\end{split}
		\end{equation}
		where we use 
		\begin{equation}\label{pf of thm_6}
			\bigg| \p_\tau^\ell  \LC \LC e^{-\tau \mathcal{H}_1}  -e^{-\tau \mathcal{H}_2}  \RC f\RC (x,t) \bigg| \leq C e^{-\frac{c}{\tau}}
		\end{equation}
		for any $k,\ell \in \N$, for any $\tau \in (0,1)$ and for some constant $c>0$.
		
		With \eqref{pf of thm_5} at hand, integrating \eqref{pf of thm_4} $m$-times, there holds that 
		\begin{equation}
			\begin{split}
				\int_0^\infty  \LC \LC e^{-\tau \mathcal{H}_1}  -e^{-\tau \mathcal{H}_2}  \RC f\RC (x,t) \frac{d\tau}{\tau^{m+1-s}}=0, \text{ for all }m\in \N,
			\end{split}
		\end{equation}			
		which  is equivalent to 
		\begin{equation}\label{pf of thm_7}
			\begin{split}
				\int_0^\infty  \LC \LC e^{-\tau \mathcal{H}_1}  -e^{-\tau \mathcal{H}_2}  \RC f\RC (x,t) \frac{d\tau}{\tau^{m+2-s}}=0, \text{ for all }m\in \N\cup \{0\}.
			\end{split}
		\end{equation}	
		By using the change of variable that $\eta=\frac{1}{\tau}$, the identity \eqref{pf of thm_7} yields that 
		\begin{equation}\label{pf of thm_8}
			\int_0^\infty \varphi (\eta) \eta^m \, d\eta  =0 ,
		\end{equation}
		for all $m\in \N$, where 
		\begin{equation}
			\varphi(\eta):=  \frac{\big( \big( e^{-\frac{1}{\eta}\mathcal{H}_1}  -e^{-\frac{1}{\eta} \mathcal{H}_2}  \big) f\big)(x,t)}{\eta^s}, \text{ for }(x,t)\in \omega_2 \times \R.
		\end{equation}
		By using \eqref{expoenential decay} and \eqref{pf of thm_6}, one can obtain that 
		\begin{equation}
			\abs{\varphi(\eta)}\leq C \frac{e^{-c\eta}}{\eta^s},
		\end{equation}
		for some constants $c,C>0$. Therefore, the Fourier transform of $\chi_{[0,\infty)}(\eta)\varphi(\eta)$ is 
		\begin{equation}
			\mathcal{F}\LC \chi_{[0,\infty)}(\eta)\varphi(\eta)\RC (\xi) = \frac{1}{\sqrt{2\pi}}\int_0^\infty \varphi(\eta)e^{-\mathsf{i}\eta\xi}\, d\eta
		\end{equation}
		is holomorphic for $\mathrm{Im} \xi <c$. Via \eqref{pf of thm_8}, it is known that $ \mathcal{F}\LC \chi_{[0,\infty)}\varphi\RC (\xi)$ equals to zero at $\xi =0$, which implies that $\varphi(\eta)\equiv 0$ for $\eta>0$. This implies that 
		\begin{equation}\label{pf of thm_9}
			\LC \LC e^{-\tau \mathcal{H}_1}  -e^{-\tau \mathcal{H}_2}  \RC f\RC (x,t)=0 \text{ for }(x,t)\in \omega_2 \times \R \text{ and }\tau>0.
		\end{equation}
		The next task is to transfer the above identity into another useful identity \eqref{same heat kernel} as we mentioned.\\

		Thanks to Lemma \ref{Lemma: double heat equation}, we know that $e^{-\tau \mathcal{H}_j} f$ solves the equation 
		\begin{equation}
			\LC \p_\tau -\Delta_{g_j} \RC  \LC e^{-\tau \mathcal{H}_j} f\RC = -\p_t \LC e^{-\tau \mathcal{H}_j} f\RC, \text{ for }(x,t,\tau)\in M\times \R\times (0,\infty),
		\end{equation}
		for $j=1,2$. Let us integrate the above equation with respect to $t\in \R$, so that 
		\begin{equation}\label{pf of thm_10}
			\begin{split}
				&\quad \, \LC \p_\tau -\Delta_{g_j} \RC F_j \\
				&= -\int_{-\infty}^\infty\p_t \LC e^{-\tau \mathcal{H}_j} f\RC dt \\
				&= \lim_{t\to -\infty} \bigg( \int_M e^{-\tau \mathcal{L}_j}(x,z) f(z,t-\tau)\, dz \bigg)  -  \lim_{t\to \infty} \bigg( \int_M e^{-\tau \mathcal{L}_j}(x,z) f(z,t-\tau)\, dz \bigg)\\
				&=0,
			\end{split}
		\end{equation}
		for all $x \in M$, $\tau >0$ and $f \in C^\infty_c( \omega_1\times (-T,T) )$, where $\mathcal{L}_j:=-\Delta_{g_j}$, and 
		$$
		F_j=F_j(x,\tau)=\int_{-\infty}^\infty \LC e^{-\tau \mathcal{H}_j} f\RC dt, \text{ for }j=1,2.
		$$
		In addition, observe the function $F_j(x,\tau)$ more carefully, one can rewrite $F_j$ in terms of 
		\begin{equation}\label{pf of thm_11}
			\begin{split}
				F_j  (x,\tau)&=\int_{-\infty}^\infty \LC e^{-\tau \mathcal{H}_j} f\RC (x,t)\, dt \\
				&= \int_{-\infty}^\infty \int_M e^{-\tau\mathcal{L}_j}(x,z) f(z,t-\tau)\, dz dt \\
				&\underbrace{= \int_M e^{-\tau\mathcal{L}_j}(x,z) \LC \int_{-\infty}^\infty f(z,t-\tau)\, dt \RC dz}_{\text{By the Fubini (Tonelli) theorem}}\\
				&= \LC e^{-\tau \mathcal{L}_j} \mathbf{f}\RC (x),
			\end{split}
		\end{equation}
		for $j=1,2$, which are the well-defined, since the right-hand side in \eqref{pf of thm_11} is the heat kernel representation formula as $\mathbf{f}\in C^\infty_c(\omega_1)$. Here $e^{-\tau\mathcal{L}_j}(x,z)\geq 0$ is the heat kernel of $\p_\tau +\mathcal{L}_j$ for $j=1,2$, and 
		$$
		\mathbf{f}(z):=\int_{-\infty}^\infty f(z,t-\tau)\, dt\in C^\infty_c(\omega_1),
		$$ 
		since $f\in C^\infty_c(\omega_1 \times (-T,T))$. 
		Due to arbitrariness of $f(z,t-\tau)\in C^\infty_c(\omega_1\times (-T,T))$, the function $\mathbf{f}(z)\in C^\infty_c(\omega_1)$ can be also arbitrary\footnote{This can be easily seen by the following simple observations: Given any $\mathbf{f}\in C^\infty_c(\omega_1)$, choosing $f(z,t-\tau):=\mathbf{f}(z)h(t-\tau)$, where $h(\cdot) \in C^\infty_c(-T,T)$ such that $\int_{-\infty}^\infty h(t)\, dt= \int_{-\infty}^\infty h(t-\tau)\, dt=1$, for all $\tau>0$.}.

		Now, on the one hand, integrate \eqref{pf of thm_9} with respect to $t\in (-\infty,\infty)$, plugging \eqref{pf of thm_11} into \eqref{pf of thm_9} (after $t$-integration), then there holds that 
		\begin{equation}
			\LC \LC e^{-\tau \mathcal{L}_1}  -e^{-\tau \mathcal{L}_2}  \RC \mathbf{f}\RC (x)=0 \text{ for }x\in \omega_2 \text{ and }\tau>0.
		\end{equation}
		On the other hand, without loss of generality, we may assume that $\mathcal{O}$ is connected, which is contained in a single coordinate patch for both manifolds $M_1$ and $M_2$. Note that the function $e^{-\tau \mathcal{L}_j}\mathbf{f}$ solves the heat equation \eqref{pf of thm_10}, for $j=1,2$, which implies that $ \LC e^{-\tau \mathcal{L}_1}  -e^{-\tau \mathcal{L}_2}  \RC \mathbf{f}$ solves the heat equation 
		\begin{equation}
			\LC \p_\tau -\Delta_g \RC  \LC \LC e^{-\tau \mathcal{L}_1}  -e^{-\tau \mathcal{L}_2}  \RC \mathbf{f} \RC=0 \text{ for }(x,\tau)\in \mathcal{O}\times (0,\infty).
		\end{equation}
		Then the unique continuation property of the heat equation (cf. \cite{Lin1990}) implies
		\begin{equation}
			\LC \LC e^{-\tau \mathcal{L}_1}  -e^{-\tau \mathcal{L}_2}  \RC \mathbf{f} \RC =0 \text{ for }(x,\tau)\in \mathcal{O}\times (0,\infty).
		\end{equation}
		Now, as $f(x,t)\in C^\infty_c(\omega_1\times (-T,T))$ and $\omega_1\Subset \mathcal{O}$ arbitrary, so $\mathbf{f}(x)\in C^\infty_c(\omega_1)$ and $\omega_1$ can be also arbitrary.

		Therefore, one can obtain 
		\begin{equation}
			\left. e^{-\tau \mathcal{L}_1} \mathbf{f} \right|_{\mathcal{O}}=\left.	e^{-\tau \mathcal{L}_2} \mathbf{f}\right|_{\mathcal{O}}, \text{ for }\tau>0.
		\end{equation}
		Due to the arbitrariness of $\mathbf{f}\in C^\infty_c(\mathcal{O})$, the above identity implies 
		\begin{equation}
			e^{-\tau \mathcal{L}_1}(x,z)=e^{-\tau \mathcal{L}_2}(x,z), \text{ for }x,z \in \mathcal{O} \text{ and }\tau>0,
		\end{equation}
		i.e., \eqref{same heat kernel} holds. Finally, applying \cite[Theorem 1.5]{feizmohammadi2021fractional}, one can find a diffeomorphism $\Phi:M_1\to M_2$ such that $\Phi^\ast g_2=g_1$ on $M_1$ with $\Phi=\mathrm{Id}$ on $\mathcal{O}$. This proves the assertion.		
	\end{proof}

	\bigskip

	\noindent\textbf{Acknowledgments.} 
	The author is partially supported by the National Science and Technology Council (NSTC) Taiwan, under the projects 113-2628-M-A49-003 \& 113-2115-M-A49-017-MY3. The author is also a Humboldt research fellow (for experienced researchers) in Germany.

	\bibliography{refs} 
	
	\bibliographystyle{alpha}

\end{document}